\documentclass[reqno,12pt,letterpaper]{amsart}
    
  \RequirePackage{amsmath,amssymb,amsthm,graphicx,mathrsfs,url}
  \RequirePackage[usenames,dvipsnames]{color}
  \RequirePackage[colorlinks=true,linkcolor=Red,citecolor=Green]{hyperref}
  \RequirePackage{amsxtra}

\title[Noncompactly supported Breit-Wigner series]{The Breit-Wigner series for noncompactly supported potentials on the line}
\author{Aidan Backus}
\date{May 2020}

\newcommand{\NN}{\mathbf{N}}

\newcommand{\RR}{\mathbf{R}}
\newcommand{\CC}{\mathbf{C}}

\DeclareMathOperator{\ch}{ch}

\DeclareMathOperator{\Res}{Res}

\DeclareMathOperator{\supp}{supp}
\newcommand{\tr}{\operatorname{tr}}

\newcommand{\dfn}[1]{\emph{#1}\index{#1}}

\renewcommand{\Im}{\operatorname{Im}}

\newtheorem{theorem}{Theorem}[section]

\newtheorem{lemma}[theorem]{Lemma}
\newtheorem{proposition}[theorem]{Proposition}

\newtheorem{conjecture}[theorem]{Conjecture}

\theoremstyle{definition}
\newtheorem{definition}[theorem]{Definition}

\newtheorem*{ack}{Acknowledgements}
\newtheorem*{notate}{Notation}

\usepackage[backend=bibtex,style=alphabetic,maxcitenames=50,maxnames=50]{biblatex}
\renewbibmacro{in:}{}
\DeclareFieldFormat{pages}{#1}

\begin{document}
\begin{abstract}
We propose a conjecture stating that for resonances, $\lambda_j$, of a noncompactly supported potential, the series $\sum_j \Im \lambda_j/|\lambda_j|^2$ diverges. This series appears in the Breit-Wigner approximation for a compactly supported potential, in which case it converges. We provide heuristic motivation for this conjecture and prove it in several cases.
\end{abstract}

\maketitle

\section{Introduction and Conjectures}
In this note we propose a conjecture on the asymptotic distribution of scattering resonances of a one-dimensional Schr\"odinger equation with a noncompactly supported, super-exponentially decreasing potential. The conjecture is motivated by the Breit-Wigner formula for compactly supported potentials. We prove this conjecture for a large class of potentials, including any analytic potential for which a conjecture of Froese \cite[Conjecture 1.2]{froese1997asymptotic} holds.

Scattering resonances are by definition the poles of the meromorphic continuation of the resolvent family $R_V(\lambda) = (-D^2 + V - \lambda^2)^{-1}$. They also may be viewed as poles of the scattering matrix $S(\lambda)$.
We let $\Res V$ be the multiset of resonances of $V$, counted with multiplicity.

The operator $-iS'(\lambda)S^*(\lambda)$ is known as the \dfn{Eisenbud-Wigner time-delay operator}, which has physical significance \cite{jensen1981time}.
In the case of compactly supported potentials, the Breit-Wigner approximation relates the trace of the Eisenbud-Wigner operator of a compactly supported potential to a sum over resonances.
\begin{theorem}[Breit-Wigner approximation for compactly supported potentials]
Suppose that $V$ is compactly supported and $\lambda_0 \in \RR$. Then the series
\begin{equation}
\label{Breit-Wigner series}
\sum_{\lambda \in \Res V \setminus 0} \frac{|\Im \lambda|}{|\lambda - \lambda_0|^2} < \infty
\end{equation}
converges, and if $V$ is real-valued then we have
\begin{equation}
\label{Breit-Wigner formula}
\frac{1}{2\pi i} \tr S'(\lambda_0) S(\lambda_0)^* = -\frac{1}{\pi}|\ch \supp V| - \frac{1}{2\pi}\sum_{\lambda \in \Res V \setminus 0}\frac{\Im \lambda}{|\lambda - \lambda_0|^2}.
\end{equation}
\end{theorem}
Here $|\ch\supp V|$ is the length of the convex hull of $\supp V$.
For a proof, see \cite[Theorem 2.20]{dyatlov2019mathematical} or \cite[Theorem 3.24]{backus2020conjecture}.
For a higher-dimensional generalization, see \cite{gerard1989breit}, \cite{petkov1999breit} and \cite{petkov2001semi}, or \cite{bruneau2003meromorphic}.

\begin{definition}
The \dfn{Breit-Wigner series} of an arbitrary potential $V$ is
$$B(V) = -\sum_{\lambda \in \Res V \setminus 0} \frac{\Im \lambda}{|\lambda|^2}.$$
\end{definition}
By (\ref{Breit-Wigner series}), $B(V)$ converges if $V$ is compactly supported.

The left-hand side, $\tr S'(\lambda)S^*(\lambda)$, of the Breit-Wigner formula (\ref{Breit-Wigner formula}) is a robust object that can be defined for a large class of decaying potentials $V$. Moreover, $\tr S'S^*$ depends continuously on $V$ in any reasonable topology, and it is not really affected by the support of $V$ as such.
Meanwhile, the right-hand side of (\ref{Breit-Wigner formula}) has a term, $|\ch \supp V|$, which is infinite when $V$ is not compactly supported, and an infinite series, so one can ask whether the right-hand side demonstrates a sort of ``cancellation of infinities." Thus, it is natural to ask whether the convergence of the Breit-Wigner series (\ref{Breit-Wigner series}) still holds when $V$ decays but is not compactly supported.

\begin{definition}
The potential $V$ is \dfn{super-exponentially decreasing} if for every $N \in \NN$, $|V(x)| \lesssim_N e^{-N|x|}$.
\end{definition}

If $V$ is a super-exponentially decreasing potential, then resonances may viewed as the zeroes of the determinant $\det(1+\sqrt V R_0\sqrt{|V|})$ \cite[\S3]{froese1997asymptotic}, and so depend continuously on the behavior of $V$ in compact sets. However, resonances may escape to infinity or otherwise be badly behaved globally. Therefore we cannot conclude that we can take the limit of the Breit-Wigner formula as the support becomes unbounded.
Yet, heuristically, one would hope that the Breit-Wigner series of a super-exponentially decreasing potential is a limit of Breit-Wigner series of compactly supported approximations. Moreover, in view of the stability of the left-hand side, we expect that as $|\ch \supp V| \to \infty$, $B(V) \to \infty$ as well, to achieve the aforementioned ``cancellation of infinities." Hence, we make the following bold conjecture.
\begin{conjecture}
\label{strong conjecture}
Let $V$ be a super-exponentially decreasing potential. The Breit-Wigner series $B(V)$ converges if and only if $V$ is compactly supported.
\end{conjecture}
The conjecture can be verified in some cases where resonances can be defined, yet the potential is not super-exponentially decreasing.
An example is the P\"oschl-Teller well,
$$V(x) = -\frac{2}{\cosh^2(x)}.$$
Its resonances are given by $-i(n+2)$, $n \in \NN$ \cite{cevik_2016}, and so $B(V)$ diverges, yet $V$ is not super-exponentially decreasing.

The distribution of $\Res V$ is in general quite difficult to study.
However, Froese made a conjecture \cite[Conjecture 1.2]{froese1997asymptotic} about the growth of the counting function of $\Res V$, and proved that a large class of potentials, including Gaussians, satisfy his conjecture \cite[Theorem 1.3]{froese1997asymptotic}.

To state Froese's conjecture, we assume that $V$ is super-exponentially decreasing, so that its Fourier-Laplace transform $\widehat V$ is entire, and introduce the following new entire function.
\begin{definition}
Given a super-exponentially decreasing potential $V$, its \dfn{Froese function} $F$ is given by
\begin{equation}
\label{froese function}
F(z) = \widehat V(2z) \widehat V(-2z) + 1.
\end{equation}
\end{definition}
We also recall the following classical definitions \cite[p. 52, p. 139]{levin1964distribution}.
\begin{definition}
Let $f$ be an entire function of order $\rho$ and normal type (that is, nonzero finite type). The \dfn{indicator function} $h$ of $f$ is given by
\begin{equation}
\label{h definition}
h(\theta) = \limsup_{r \to \infty} \frac{\log|f(re^{i\theta})|}{r^\rho}.
\end{equation}
\end{definition}
\begin{definition}
\label{completely regular growth}
Let $f$ be an entire function of order $\rho$ and normal type. If there is a subset $I$ of $\{r:r>0\}$ of density one such that for every $\theta$, the $\limsup$ appearing in (\ref{h definition}) is actually a uniform limit as $r \to \infty$ along $I$, then $f$ is said to have \dfn{completely regular growth}.
\end{definition}
Henceforth we let $A(R, \theta, \varphi)$ denote the sector
$$A(R, \theta, \varphi) = \{re^{i\alpha} \in \CC: r \leq R \text{ and } \alpha \in [\theta, \varphi]\}.$$
We let $n(R, \theta, \varphi)$ denote the number of resonances in $A(R, \theta, \varphi)$ and let $N(R, \theta, \varphi)$ denote the number of zeroes of the Froese function $F$ in $A(R, \theta, \varphi)$. We let $n(R) = n(R, 0, 2\pi)$ and similarly for $N(R)$.
With this background in place, we may recall Froese's conjecture.
\begin{conjecture}[Froese]
Suppose that $V$ is super-exponentially decreasing and $\widehat V$ has completely regular growth. Then in the lower-half plane $\CC_-$, one has
\begin{equation}
\label{froese estimate}
|n(R, \theta, \varphi) - N(R, \theta, \varphi)| = o(R^\rho).
\end{equation}
\end{conjecture}
In view of Froese's conjecture, we formulate a weaker form of Conjecture \ref{strong conjecture} as follows:
\begin{conjecture}
\label{weak conjecture}
Suppose that $V$ meets the hypotheses of Froese's conjecture and $V$ is not compactly supported. Then either $B(V)$ diverges, or $V$ is a counterexample to Froese's conjecture.
\end{conjecture}
Froese's conjecture gives a linear lower bound on the resonance-counting function $n$ (Proposition \ref{linear lower bound}), so either all resonances except for a zero-density set are contained in arbitrarily small sectors around $\RR$, or $B(V)$ diverges (Lemma \ref{divergence of angular series}).
So, if $B(V)$ converges and Froese's conjecture holds, then a positive-density set of resonances is contained in arbitrarily small sectors around $\RR$, a result that was already proven for compactly supported potentials by Zworski \cite{zworski1987distribution}.
The method of complex scaling rules this possibility out if $V$ is holomorphic in a conic neighborhood of $\RR$ \cite[Corollary 12.14]{sjostrand2002lectures}. We show that certain unnatural hypotheses on the monotone, nonnegative function
\begin{equation}
\label{s formula}
s(\theta, \varphi) = h'(\varphi) - h'(\theta) + \rho^2 \int_\theta^\varphi h(\alpha)~d\alpha,
\end{equation}
where $h$ is the indicator function of $F$, will also rule out this possibility (Theorem \ref{divergence of breit wigner, preliminary version}).
\begin{theorem}
\label{divergence of breit wigner, preliminary version}
Suppose that $V$ meets the hypotheses and conclusion of Froese's conjecture. If $V$ is noncompactly supported, then the Breit-Wigner series $B(V)$ will diverge provided that any one of the following criteria are true:
\begin{enumerate}
\item The set of resonances of $V$ contained in arbitrarily small sectors around $\RR$ is of zero density. \label{resonances in sectors}
\item $V$ is holomorphic in a conic neighborhood of $\RR$. \label{holomorphic potential}
\item There are $\theta \leq \varphi$ such that $0,\pi \notin [\theta, \varphi]$ and $s(\theta, \varphi) \neq 0$. \label{s is not defined}
\item There is a $k \in \{0, 1\}$ and a $\theta > k\pi$ such that $s(k\pi, \theta)$ exists. \label{limit of s}
\end{enumerate}
\end{theorem}
Here $s$ is given by (\ref{s formula}), and Case \ref{s is not defined} includes the possibility that $s(\theta, \varphi)$ does not exist.
We prove Theorem \ref{divergence of breit wigner, preliminary version} in Section \ref{divergence section}.

In Section \ref{linear growth}, we recall properties of the Froese function $F$ and prove the following proposition, which will be used in Section \ref{divergence section} and may be of independent interest:
\begin{proposition}
\label{linear lower bound}
Suppose that $V$ meets the hypotheses and conclusion of Froese's conjecture. Let $\rho$ denote the order of $\widehat V$. If $V$ is not identically zero, then as $r \to \infty$, $n(r) \gtrsim r^\rho \geq r$.
\end{proposition}

\begin{notate}
We will write $f \lesssim g$ to mean that there is a constant $C > 0$ such that for every $x$ such that $|x|$ is large enough, $f(x) \leq Cg(x)$.
We write $f \sim g$ to mean $g \lesssim f \lesssim g$, and use a subscript $\lesssim_y$ to mean that $C$ is allowed to depend on $y$.

Given a fixed set $I \subseteq \{r: r > 0\}$ of density one, which will always be the set $I$ that appears in Definition \ref{completely regular growth}, we write $f \approx g$ to mean that $f(r)/g(r) \to 1$ as $r \to \infty$ along $I$ (and uniformly in all other variables). We write $f \approx 0$ to mean that $f \to 0$, the limit taken along $I$.

We write $f'(x\pm 0)$ to mean the left ($-$) and right ($+$) derivatives of a semidifferentiable function $f$. We write $x_+$ to mean $\max(x, 0)$.
\end{notate}

\begin{ack}
I would like to thank Maciej Zworski for introducing me to scattering theory and for many helpful discussions; in particular, he suggested that Froese's conjecture could be used to prove certain cases of the main conjecture.

I also gratefully acknowledge partial support by the National Science Foundation under the grant DMS-1500852.
\end{ack}

\section{Linear growth of resonances}
\label{linear growth}
The following properties of the Froese function $F$ follow from its definition (\ref{froese function}) and the assumption that $\widehat V$ is an entire function of completely regular growth:
\begin{enumerate}
\item $F$ has completely regular growth.
\item The order of $F$ is $\rho$.
\item For every $z \in \CC$, $F(z) = F(-z)$.
\end{enumerate}
Let $h$ be the indicator function of $F$, and let $s$ be given by (\ref{s formula}).
We recall a characterization of $s$ \cite[Theorem III.3]{levin1964distribution}.
\begin{theorem}
\label{zeroes of entire functions}
If $\widehat V$ is an entire function of completely regular growth, then there is a countable, possibly empty, exceptional set $Z$ of angles such that:
\begin{enumerate}
\item If $\theta, \varphi \in [0, 2\pi] \setminus Z$, then
$$s(\theta, \varphi) = 2\pi\rho \lim_{r\to\infty} \frac{N(r, \theta, \varphi)}{r^\rho}.$$
\item $\theta \in Z$ if and only if $h'(\theta - 0) \neq h'(\theta + 0)$.
\end{enumerate}
\end{theorem}
Note that $h$ is semidifferentiable, hence continuous.
Moreover, $s(\theta, \cdot)$ is increasing and nonnegative for any $\theta$, and dually, $s(\cdot, \varphi)$ is decreasing and nonnegative for any $\varphi$.
If $\theta, \varphi \notin Z$, then $s(\theta, \varphi)$ must exist.

We adopt the convention that if $\theta \notin Z$ then $s(\theta, \cdot)$ is right-continuous, viz.
$$s(\theta, \varphi) = h'(\varphi + 0) - h'(\theta) + \rho^2\int_\theta^\varphi h(\alpha)~d\alpha.$$
Thus $s(\theta, \cdot)$ is defined and right-continuous on all of $[0, 2\pi]$.

\begin{lemma}
\label{characterizing jump discontinuity}
If $\widehat V$ is an entire function of completely regular growth, $\beta < \theta$, and $\beta \notin Z$, then the following are equivalent:
\begin{enumerate}
\item $s(\beta, \cdot)$ has a jump discontinuity at $\theta$.
\item $s(\beta, \cdot)$ is not continuous at $\theta$.
\item $\theta \in Z$.
\end{enumerate}
\end{lemma}
\begin{proof}
First observe that since $s(\beta, \cdot)$ is monotone, all discontinuities are jump discontinuities.

Suppose that $s(\beta, \cdot)$ is continuous at $\theta$.
Thus
$$\lim_{\delta \to 0} s(\theta - \delta, \theta + \delta) = \lim_{\delta \to 0} s(\beta, \theta + \delta) - s(\beta, \theta - \delta) = 0,$$
the limit taken along $\delta > 0$ such that $\theta - \delta \notin Z$. Yet
$$s(\theta - \delta, \theta + \delta) = h'(\theta - \delta) - h'(\theta + \delta) + \rho^2 \int_{\theta - \delta}^{\theta + \delta} h(\alpha)~d\alpha,$$
and taking the limit of both sides as $\delta \to 0$ we see that $h'(\theta - 0) = h'(\theta + 0)$, so $\theta \notin Z$.

Conversely, if $s(\beta, \cdot)$ has a jump discontinuity at $\theta$ then
$$0 < \lim_{\varepsilon \to 0} s(\beta, \theta) - s(\beta, \theta - \varepsilon) = h'(\theta + 0) - h'(\theta - 0)$$
so $\theta \in Z$.
\end{proof}

\begin{lemma}
\label{zeroes of entire functions with jump discontinuity}
If $\widehat V$ is an entire function of completely regular growth and either $s(\theta, \varphi)$ is nonzero or $\theta \in Z$, then
$$N(r, \theta, \varphi) \sim r^\rho.$$
\end{lemma}
\begin{proof}
This follows immediately from Theorem \ref{zeroes of entire functions} if $\theta \notin Z$ and $s(\theta, \cdot)$ is continuous at $\varphi$.
So suppose otherwise.

If $\theta \in Z$, then let $\beta < \theta$, $\beta \notin Z$.
Then by Lemma \ref{characterizing jump discontinuity}, $s(\beta, \cdot)$ has a jump discontinuity at $\theta$, say by $\eta > 0$.
For every $\varepsilon > 0$ small enough,
$$s(\theta - \varepsilon, \varphi) = s(\beta, \varphi) - s(\beta, \theta - \varepsilon) \geq s(\beta, \theta) - s(\beta, \theta - \varepsilon) \geq \eta.$$
Thus
$$N(r, \theta - \varepsilon, \varphi) \geq \frac{\eta}{\rho}r^\rho$$
if $r$ is large enough, uniformly in $\varepsilon$. Since
$$A(r, \theta, \varphi) = \bigcap_{\varepsilon > 0} A(r, \theta - \varepsilon, \varphi)$$
is a closed sector, it follows that
$$N(r, \theta, \varphi) \geq \frac{\eta}{\rho}r^\rho$$
if $r$ is large enough. This proves the lemma in the case $\theta \in Z$.

Thus we may assume that $\theta \notin Z$ and $s(\theta, \cdot)$ is discontinuous at $\varphi$.
If this happens, choose $\varphi' > \varphi$ such that $\varphi' - \varphi$ is small and $\varphi' \notin Z$. Then
$$s(\theta, \varphi') \geq s(\theta, \varphi) > 0,$$
so by Theorem \ref{zeroes of entire functions}, $N(r, \theta, \varphi') \sim r^\rho$ uniformly in $\varphi'$, again proving the lemma.
\end{proof}

\begin{lemma}
\label{no trigonometric indicators}
If $\widehat V$ is an entire function of completely regular growth, then $h \geq 0$.
\end{lemma}
\begin{proof}
Let $H$ be the indicator function of $\widehat V$.
Since $\widehat V$ has completely regular growth,
$$\log |\widehat V(re^{i\theta})| \approx H(\theta)r^\rho.$$
Moreover, if $T$ is any continuous function and $f \approx g$ then $T(f) \approx T(g)$, so $|\widehat V(re^{i\theta})| \approx \exp(r^\rho H(\theta))$ and hence
$$|F(re^{i\theta})| \approx 1 + \exp(2^\rho r^\rho(H(\theta) + H(\pi + \theta))).$$
Therefore
$$\log |F(re^{i\theta})| \approx \begin{cases}
0 &\text{ if $H(\theta) + H(\pi + \theta) < 0$}\\
\log 2 &\text{ if $H(\theta) + H(\pi + \theta) = 0$}\\
2^\rho r^\rho (H(\theta) + H(\pi + \theta)) &\text{ else}
\end{cases}$$
so
$$h(\theta) = 2^\rho (H(\theta) + H(\pi + \theta))_+ \geq 0,$$
which completes the proof.
\end{proof}

\begin{proof}[Proof of Proposition \ref{linear lower bound}]
We first remark that $\rho \geq 1$, a consequence of the Paley-Wiener-Schwartz theorem.
Indeed, if $V$ is compactly supported, then $\rho = 1$; otherwise, either $\rho > 1$ or the type of $V$ is $0$; the latter is excluded by Definition \ref{completely regular growth}.

By Froese's conjecture and Lemma \ref{zeroes of entire functions with jump discontinuity}, it suffices to show that either $\pi \in Z$ or there is an angle $\theta \in [\pi, 2\pi]$ such that $s(\pi, \theta)$ is nonzero.

To do this, we first show that $s(0, \cdot)$ is not identically zero. Suppose that it is. Then
$$
h'(\varphi) = h'(\theta) + \rho^2 \int_\varphi^\theta h(\alpha) ~d\alpha,
$$
yet $h$ is continuous and $\theta$ is fixed, so $h' \in C^1$ and so $h^{(2)} = -\rho^2 h$,
so there are constants $c_\pm$ such that
$$h(\varphi) = c_+e^{i\rho\varphi} + c_-e^{-i\rho\varphi}.$$
Since $F$ has completely regular growth, $F$ is of normal type, so $h$ is not identically zero.
Since $h$ is real-valued, this implies that $h$ has a simple zero in $(0, 2\pi)$.
Therefore $h$ is not nonnegative, contradicting Lemma \ref{no trigonometric indicators}.


So either $0 \in Z$ or there is an angle $\theta \in [0, 2\pi]$ such that either $s(0, \theta) \neq 0$.
Using the reflection symmetry $F(z) = F(-z)$, either $\pi \in Z$ or we may replace $\theta$ with a $\theta \in [\pi, 2\pi]$ such that $s(\pi, \theta) \neq 0$, if necessary.
\end{proof}

\section{Divergence of $B(V)$}
\label{divergence section}
Assume that $V$ is noncompactly supported; we are ready to prove that $B(V)$ diverges. We recall that there were four sufficient conditions to check; any one would imply that $B(V)$ diverges.
But Case \ref{holomorphic potential} reduces to Case \ref{resonances in sectors}: if $V$ is holomorphic in a conic neighborhood of $\RR$, then there are only finitely many resonances in a conic neighborhood of $\RR$ \cite[Corollary 12.14]{sjostrand2002lectures}.

Similarly, Case \ref{limit of s} reduces to Case \ref{resonances in sectors}: if $s(k\pi, \theta)$ exists, then $h$ is differentiable at $k\pi$, so $k\pi \notin Z$;
then Lemma \ref{characterizing jump discontinuity} implies that if $\beta < k\pi$ then $s(\beta, \cdot)$ is continuous at $k\pi$ and hence
$$\lim_{\varepsilon \to 0} \lim_{r \to \infty} \frac{N(r, k\pi - \varepsilon, k\pi + \varepsilon)}{r^\rho} = \frac{1}{2\pi\rho} \lim_{\varepsilon \to 0} s(k\pi - \varepsilon, k\pi + \varepsilon) = 0,$$
by Theorem \ref{zeroes of entire functions}. Since $N(r) \sim r^\rho$ by Proposition \ref{linear lower bound}, this implies Case \ref{resonances in sectors}.
%

\begin{lemma}
All but finitely many resonances of $V$ are in the lower-half plane $\CC_-$.
\end{lemma}
\begin{proof}
This is well-known, but we sketch the proof; see \cite[\S3]{froese1997asymptotic} or \cite[Lemma 3.23]{backus2020conjecture} for the details.
Let $\mathcal B^1(\mathcal H)$ denote the trace class of $\mathcal H = L^2(\supp V)$. Choosing an appropriate branch of $\sqrt\cdot$, we may identify resonances with the zeroes of the function
$$D(\lambda) = \det(1 + \sqrt V R_0(\lambda) \sqrt{|V|}),$$
which is holomorphic in the upper-half plane $\CC_+$ since $\sqrt V R_0 \sqrt{|V|}$ is holomorphic $\CC_+ \to \mathcal B^1(\mathcal H)$. Moreover, $D(\lambda) \to 1$ as $\lambda \to \infty$ along any ray in $\CC_+$,
so there are only finitely many zeroes of $D$ in $\CC_+$.
\end{proof}
Therefore we may replace $B(V)$ with a sum over \emph{only the resonances in $\CC_-$} without affecting its convergence properties, so that all summands in $B(V)$ are positive.

\begin{lemma}
\label{divergence of angular series}
Suppose that $\pi < \theta \leq \varphi < 2\pi$. If $n(r, \theta, \varphi) \gtrsim r$, then $B(V)$ diverges.
\end{lemma}
\begin{proof}
Let $k_j = n(j, \theta, \varphi)$, so that $k_j \gtrsim j$.
Let $\Res^* V$ be the set of resonances $re^{i\xi}$ such that $\theta \leq \xi \leq \varphi$. Then
\begin{align*}
B(V) &\geq -\sum_{\lambda \in \Res^* V} \frac{\Im \lambda}{|\lambda|^2}  \geq \min(-\sin \theta, -\sin \varphi) \sum_{\lambda \in \Res^* V} \frac{1}{|\lambda|}\\
  & \gtrsim_{\theta,\varphi} \sum_{j=0}^\infty \sum_{\substack{\lambda \in \Res^* V\\|\lambda| \in [j, j+1)}} \frac{1}{j+1}
  = \sum_{j=0}^\infty \frac{k_{j+1} - k_j}{j+1}.
\end{align*}
Summing by parts,
\begin{align*}
\sum_{j=0}^J \frac{k_{j+1} - k_j}{j+1} &= \frac{k_{J+1}}{J+1} - \sum_{j=1}^J k_j\left(\frac{1}{j+1} - \frac{1}{j}\right)\\
&=\frac{k_{J+1}}{J+1} + \sum_{j=1}^J \frac{k_j}{j + j^2} \gtrsim 1 + \sum_{j=1}^J \frac{1}{j}
\end{align*}
which $\to \infty$ as $J \to \infty$.
\end{proof}

In Case \ref{resonances in sectors},
the resonances $\lambda$ furnished by Proposition \ref{linear lower bound} will satisfy $-\sin \arg \lambda > \delta$ for some sufficiently small $\delta > 0$,
so by Lemma \ref{divergence of angular series}, $B(V)$ diverges.

Finally, we prove Case \ref{s is not defined}. By reflection, we can assume that $\pi < \theta \leq \varphi < 2\pi$.
By Lemma \ref{zeroes of entire functions with jump discontinuity}, Froese's conjecture, and Proposition \ref{linear lower bound},
$$n(r, \theta, \varphi) \sim N(r, \theta, \varphi) \gtrsim r.$$
Thus Lemma \ref{divergence of angular series} completes the proof of Theorem \ref{divergence of breit wigner, preliminary version}.

\printbibliography

\end{document}